\documentclass[11pt,reqno]{amsart}
\usepackage{amsthm,amsmath,amssymb}
\usepackage{fullpage}
\usepackage{graphicx}
\usepackage[colorlinks=true,citecolor=black,linkcolor=black,urlcolor=blue]{hyperref}

\theoremstyle{plain}
\newtheorem{theorem}{Theorem}
\newtheorem{lemma}[theorem]{Lemma}
\newtheorem{corollary}[theorem]{Corollary}
\newtheorem{proposition}[theorem]{Proposition}

\theoremstyle{definition}
\newtheorem{definition}[theorem]{Definition}
\newtheorem{example}[theorem]{Example}
\newtheorem{conjecture}[theorem]{Conjecture}

\newtheorem{problem}[theorem]{Problem}

\theoremstyle{remark}

\newcommand{\Complex}{\mathbb{C}}
\newcommand{\CC}{\mathbb{C}}
\newcommand{\ds}{\displaystyle}
\newcommand{\Recon}[2]{{\mathsf{Recon}}(#1,#2)}
\newcommand{\MMmatrix}[3]{{\mathfrak{M}}^{#1}_{#3} {#2}}
\newcommand{\RRmatrix}[2]{{\mathfrak{R}}^{#1}_{#2}}
\def\bbbeta{z}

\DeclareMathOperator{\trace}{\mathrm{trace}}
\DeclareMathOperator{\rel}{\mathsf{rel}}
\DeclareMathOperator{\Sym}{\mathcal{S}}
\DeclareMathOperator{\Inde}{\mathsf{Ind}}
\DeclareMathOperator{\indparts}{\mathsf{indparts}}
\DeclareMathOperator{\flip}{\mathsf{flip}}

\DeclareMathOperator{\comp}{\mathsf{comp}}
\DeclareMathOperator{\Colourings}{\mathsf{Colours}}
\DeclareRobustCommand{\stirlingfirst}{\genfrac[]{0pt}{}}
\DeclareRobustCommand{\stirlingsecond}{\genfrac\{\}{0pt}{}}
\newcommand{\peg}{p}
\newcommand{\Pegs}{\mathrm{Pegs}}
\DeclareMathOperator{\PegSet}{\mathrm{PegSet}}

\def\retyup{1}

\newcommand{\RestColourings}[2]{\mathsf{Colourings}(#1,#2)}

\newcommand{\Numb}[2]{ \left(\kern-0.9ex\left(\kern-0.9ex\left(\hspace*{-0.6ex}\begin{array}{c}#1\\ #2\end{array}\hspace*{-0.6ex} \right)\kern-0.9ex\right)\kern-0.9ex\right)^{\kern-0.6ex\star} } 
\newcommand{\sNumb}[2]{ \left(\kern-0.9ex\left(\kern-0.9ex\left(\substack{#1\\[0.3em] #2} \right)\kern-0.9ex\right)\kern-0.9ex\right)^{\kern-0.6ex\star} } 

\newcommand{\OtherNumb}[2]{N_{#2}(#1) }
\newcommand{\starcup}{$\sqcup$\kern-0.58em{$\star$}}
\newcommand{\lbrak}{\left[\kern-0.6em\right(}
\newcommand{\webproduct}{\,\blacklozenge\,}
\newcommand{\webpower}[2]{#1^{\webproduct #2}}
\newcommand{\Fubini}[2]{\mathcal{F}_{#1}{#2}}

\title{Web matrices: structural properties and\\ generating combinatorial identities}
\author[M. Dukes]{Mark Dukes}
\author[C. D. White]{Chris D. White}
\address{MD: Department of Computer and Information Sciences, University of Strathclyde, Glasgow, G1 1XH, UK}
\address{CDW: School of Physics and Astronomy, University of Glasgow, Glasgow, G12 8QQ, UK}


\begin{document}

\maketitle

\begin{abstract}
In this paper we present new results for the combinatorics of web diagrams and
web worlds. These are discrete objects that arise in the physics of calculating
scattering amplitudes in non-abelian gauge theories. Web-colouring and
web-mixing matrices (collectively known as web matrices) are indexed by ordered
pairs of web-diagrams and contain information relating the number of colourings
of the first web diagram that will produce the second diagram.

We introduce the black diamond product on power series and show how it
determines the web-colouring matrix of disjoint web worlds. Furthermore, we show
that combining known physical results with the black diamond product gives a new
technique for generating combinatorial identities. Due to the complicated action
of the product on power series, the resulting identities appear highly
non-trivial.

We present two results to explain repeated entries that appear in the web matrices. 
The first of these shows how diagonal web matrix
entries will be the same if the comparability graphs of their associated
decomposition posets are the same. The second result concerns general repeated
entries in conjunction with a flipping operation on web diagrams.

We present a combinatorial proof of idempotency of the web-mixing matrices,
previously established using physical arguments only. We also show how the
entries of the square of the web-colouring matrix can be achieved by a linear
transformation that maps the standard basis for formal power series in one
variable to a sequence of polynomials. We look at one parameterized web world
that is related to indecomposable permutations and show how determining the
web-colouring matrix entries in this case is equivalent to a combinatorics on
words problem.
\end{abstract}


\section{Introduction}\label{secone}
Web diagrams are discrete objects that are subject to certain colouring and reconstruction
operations, and arise in physics in the calculation of scattering
amplitudes in non-abelian gauge theories~\cite{Gardi:2010rn,Mitov:2010rp,Vladimirov:2015fea}. 
A prominent example is the
theory of quarks and gluons, Quantum Chromodynamics (QCD), which is of
great topical relevance given its application to current experiments
such as the Large Hadron Collider. 

Whilst some properties of web-mixing matrices have been established 
from a physics point of view~\cite{Gardi:2010rn,GPO}, a fully general 
understanding of their structure and properties remains elusive. 
Web-colouring and web-mixing matrices admit a purely combinatorial definition and this provides an alternative
framework to the physical picture to elucidate their properties. 
Our ultimate hope is that a detailed understanding of web diagrams and their 
matrices can be used to dramatically improve the precision of theoretical 
predictions for particle collider experiments. 

The combinatorics of web diagrams and web matrices is an interesting study 
in its own right in that it combines parts of order theory, graph theory, and the theory of permutations in a new and novel way. 
Our seminal paper on the combinatorics of web diagrams \cite{jcta} looked at these objects from several angles
and has a companion physics paper~\cite{jhep} which explains how the results are pertinent to particle physics applications~\footnote{For a recent introduction to the latter, see ref.~\cite{White:2015wha}.}.
We were able to show, for example, that for particular web diagrams, the diagonal entries of the web matrices corresponding to them depend on the generating function of the descent statistic formed by summing over the Jordan-H\"older set of all linear extensions of a web diagram's decomposition poset (partially ordered set).
We examined some special (parameterized) classes of web worlds which could be `exactly' solved, and one of these saw the introduction of a new permutation statistic that
did not seem to have been previously studied.
Another result is that the number of different diagrams in a web world can be given by a hook-length style formula on its representation matrix.

Our aim in this study is not just to build on some of the results and directions that were initiated in \cite{jcta}, but to investigate some new aspects of these web diagrams that have not been explored before.
Amongst these, we will give a combinatorial proof of idempotency of the web-mixing matrices
that was originally established by way of a physics argument in Gardi
and White~\cite[\S 3]{GPO}.
We will also introduce the {\it{black diamond product}} on power series
and show how the web-colouring matrix of a web world whose web diagrams can be `partitioned' (in a sense to be made precise later) depend on this black diamond product.
Furthermore, we will show how a physics argument gives rise to a new technique for generating combinatorial identities using the black diamond product.
We think that this technique is interesting in that it has
the potential to generate quite unusual looking identities due to the
black diamond product behaving in quite a complicated way on the power series in question.

We will concern ourself with the web diagrams that were defined in ~\cite{jcta}~\footnote{The diagrams we consider here are
referred to as {\it Multiple Gluon Exchange Webs} (MGEWs) in
refs.~\cite{Gardi:2013saa,Gardi:2011yz}. Other types of web are
possible, but are beyond the scope of the present work.}. 
Let $P=\{p_1,\ldots,p_n\}$ be a set of $n$ pegs that are rooted to a
plane.  A web diagram $D$ on $P$ is a set of 4-tuples
$e_i=(a_i,b_i,c_i,d_i)$ where the $e_i$ represents an edge between peg
$p_{a_i}$ and peg $p_{b_i}$ for which the endpoint of $e_i$ on peg
$p_{a_i}$ is the $c_i$th highest vertex on that peg from the plane,
and the endpoint of $e_i$ on peg $b_i$ is the $d_i$th highest vertex
on that peg from the plane.  The labels of the endpoints of edges on
each peg must be distinct, and this is taken care of in the formal
definition, Definition~\ref{wddefn}.
We will sometimes abuse the terminology by referring to `peg $i$' in place of `peg $p_i$'.

Two illustrations of web diagrams are given in Figure~\ref{fig:1}.
The \emph{web world} of a web diagram $D$ is the set of all web
diagrams that result from permuting {\it{the order}} in which vertices
appear on a peg, and we usually denote it by $W(D)$.  The two web
diagrams in Figure~\ref{fig:1} are in the same web world. In the
physics context in which such diagrams arise, the pegs represent
quarks or gluons emanating from the interaction point, and the lines
joining pegs correspond to additional gluons being radiated and
absorbed. The conventional depiction of each web diagram is a
so-called {\it Feynman diagram}, an example of which is shown on the
left of Figure~\ref{fig:feyn}.

\begin{figure}[h!]
\centerline{
\begin{tabular}{c@{\hspace*{5em}}c@{\hspace*{5em}}c}
{\includegraphics[scale=0.8]{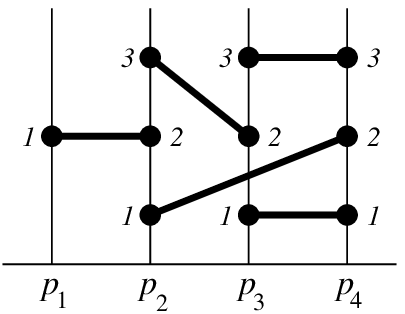}} & {\includegraphics[scale=0.8]{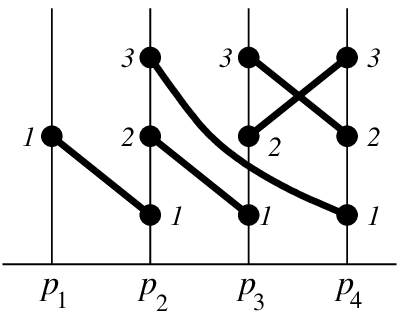}} & {\includegraphics[scale=0.9]{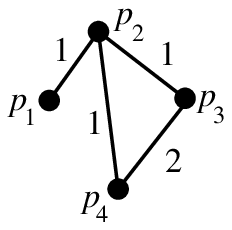}} \\
(a) & (b) & (c)
\end{tabular}
}
\caption{\label{fig:1} Two examples of web diagrams in the same web
  world. Web diagrams (a) and (b) in this figure both have a single edge
  between `pegs' $p_2$ and $p_4$.  However the heights of the
  endpoints differ in both diagrams.  In (a) the heights of this edge
  on pegs $p_2$ and $p_4$ are, respectively, 1 and 2. We represent
  this edge by the 4-tuple $(2,4,1,2)$.  In (b) the heights of this
  edge are, respectively, 3 and 1. The edge in this diagram is
  represented by the 4-tuple $(2,4,3,1)$.  The web graph of the web
  world (to be defined in Section \ref{sectwo}) containing the web diagrams is shown as (c).  }
\end{figure}

To every web world $W$ we associate two important matrices,
$\MMmatrix{(W)}{}{}$ and $\RRmatrix{(W)}{}$, called the {\it{web-colouring}}
and {\it{web-mixing}} matrices, respectively.  The entries of these
matrices are indexed by ordered pairs of web diagrams, and contain
information about colourings of the edges of the first web diagram
that yield the second diagram under a construction determined by a
colouring. 

In this paper we deepen our study of these recently defined structures, with the
goal of ascertaining general properties. Section \ref{sectwo} contains some
background and definitions necessary for the rest of the paper.  In
Section \ref{secthree} we introduce the black diamond product on formal power series and prove a decomposition theorem for the web matrices of web
worlds that may be partitioned into web worlds on disjoint peg sets.
In Section \ref{secfour} we give a new method for generating combinatorial identities by exploiting some physical properties of web worlds.

In Section \ref{secfive} we give two new results to explain the repeated entries
that one notices occurring in web matrices.  The first of theses
results concerns diagonal entries and shows that if the comparability
graphs of the decomposition posets of two web diagrams (that satisfy a
further technical condition) are the same, then the diagonal entries
on the web matrices for these two diagrams will be the same.  The
second result explains how the action of flipping a web diagram
upside-down combines with the colouring and reconstruction operations
on that same web diagram.  In Section \ref{secsix} we derive an expression for
the entries of the square of a web-colouring matrix in terms of the
entries of the web-colouring matrix, and give a combinatorial proof
that the web-mixing matrices are idempotent.

Finally, in Section \ref{secseven} we investigate a particular web world whose diagrams consist of only two pegs that have multiple edges between them. We show how they are related to indecomposable permutations and how calculating the entries of the web matrices for this web world reduces to a combinatorics on words problem.


\section{Definitions and terminology}\label{sectwo}
Let $\CC[[x]]$ be the ring of formal power series in the variable $x$ over the ring $\CC$.
Given $f\in \CC[[x]]$ we will denote by $[x^i] f$ the coefficient of $x^i$ in $f$ and we extend this notation to the multivariate case.
For integers $a,b \in \mathbb{N}$ let $a \vee b$ denote the maximum of $a$ and $b$.
If $a \leq b$, then let $[a,b]=\{a,a+1,\ldots,b\}$.

Although web diagrams were defined in Section 1, we give here the formal specification in order to remove any uncertainty surrounding their definition.

\begin{definition}\label{wddefn}
A {\it{web diagram}} on $n$ pegs having $m$ edges is a collection
$D=\{e_j=(a_j,b_j,c_j,d_j): 1\leq j \leq m\}$ of 4-tuples that satisfy
the following properties:
  \begin{enumerate}
  \item[(i)] $1\leq a_j < b_j \leq n$ for all $j \in [1,m]$.
  \item[(ii)] For $i \in [1,n]$ let $\peg_i(D)$ be the number
    of $j$ such that $a_j$ or $b_j$ equals $i$, that is, the number of
    edges in $D$ incident with peg $i$.  Then
the labels of the vertices on peg $i$ must be the first $\peg_i(D)$ positive integers, i.e.
\begin{align*}
\{d_j: b_j=i\} \cup \{c_j: a_j = i\} &= [1,\peg_i(D)],\\
\noalign{and no vertex on a peg can be both a left endpoint and a right endpoint of an edge, i.e.}
\{d_j: b_j=i\} \cap \{c_j: a_j = i\} &= \emptyset.
\end{align*}
  \end{enumerate}
  We write $\Pegs(D)=(\peg_1(D),\ldots,\peg_n(D))$.  The labels of the
  $\peg_i(D)$ vertices on peg $i$ when read from bottom to top are
  $(1,\ldots,\peg_i(D))$.
\end{definition}
Given a web diagram $D=\{e_j=(a_j,b_j,c_j,d_j): 1\leq j \leq m\}$,
the set of pegs that are incident with at least one edge is
$\PegSet(D) = \{a_1,\ldots,a_m,b_1,\ldots,b_m\}$.

\begin{figure}[h!]
\centerline{
\begin{tabular}{c@{\qquad\qquad\qquad}c}
\includegraphics[scale=0.45]{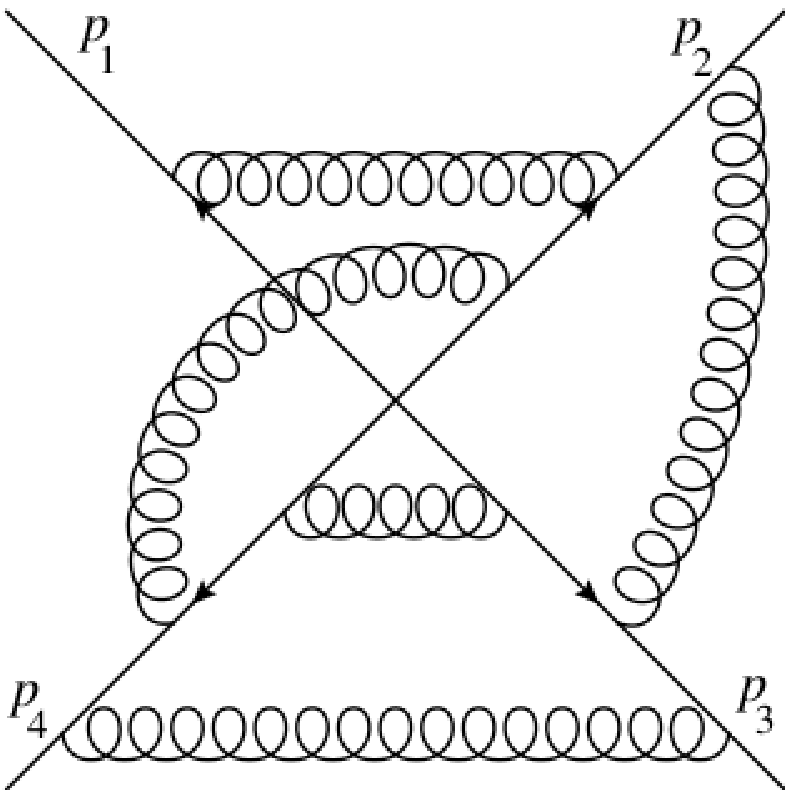} &
\includegraphics[scale=0.7]{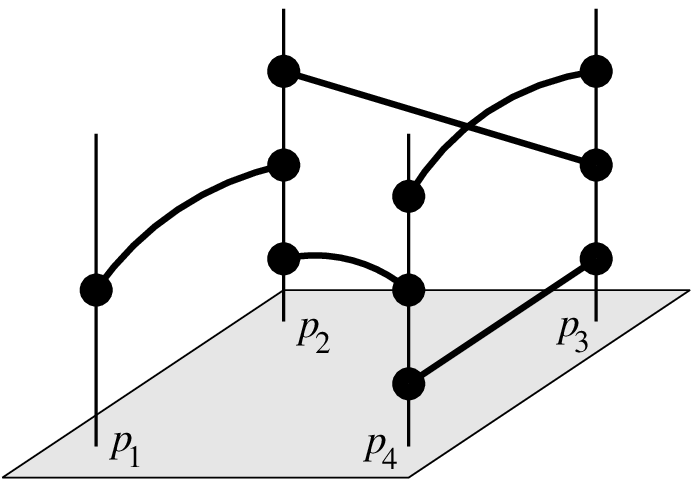} 
\end{tabular}
}
\caption{\label{fig:feyn} 
  The Feynman diagram corresponding to the web diagram of 
  Figure~\ref{fig:1}(a) is illustrated on the left. The
  gluon lines correspond to the edges of the web diagram. The heights
  of the endpoints correspond to their distance from the meeting points
  of the the 4 arrowed lines. The `pegs and plane' representation is shown to its
  right.  }
\end{figure}

A web diagram $D$ may be mapped to a simple graph $G([1,n],E)$ where $E(G)=\{\{a,b\}~:~ (a,b,c,d) \in D\}$.
Web worlds are in 1-1 correspondence with loop-free simple
graphs whose edges have labels in the set of positive integers.  
We call the labeled graph associated with a web world its {\it{web graph}}.
The vertices of the web graph correspond to the pegs of the web diagram 
(all vertices on a peg in the web diagram are projected to a 
single vertex in the web graph).
The edges of the web graph are labeled with the multiplicities of edges 
between the pegs in the web diagram.  
The web world that the web diagrams (a) and (b) of
Figure~\ref{fig:1} are members of is represented by the web graph (c).

\begin{definition}
  Let $D = \{e_j=(x_j,y_j,a_j,b_j): 1\leq j \leq m\}$ and
  $D'=\{e'_j=(x'_j,y'_j,a'_j,b'_j): 1\leq j \leq m'\}$ be two
  web diagrams with $\PegSet(D),\PegSet(D') \subseteq \{1,\ldots,n\}$.
  The {\textit{sum}} $D\oplus D'$ is the
  web diagram obtained by placing the diagram $D'$ on top of $D$;
$$
D\oplus D' ~=~ D \cup \{(x'_j,y'_j,a'_j + \peg_{x'_j}(D), b'_j +
\peg_{y'_j}(D)): 1\leq j \leq m'\}.
$$ 
If there exist two non-empty web diagrams $E$ and $F$ such that
$D=E \oplus F$ then we say that $D$ is {\it{decomposable}}.  Otherwise
we say that $D$ is {\it{indecomposable}}.
\end{definition}

\begin{example}\label{ex-diag-poset}
  Consider the following two web diagrams:
  $D_1=\{(1,4,1,1),(2,6,1,2),$ $(2,6,2,1)\}$ and
  $D_2=\{(1,2,1,1),(3,5,1,1),(5,6,2,1)\}$.  For $D_1$ we have
  $$(\peg_1(D_1),\ldots,\peg_6(D_1))=(1,2,0,1,0,2)$$ and so
  \begin{align*}
\lefteqn{D_1\oplus D_2} \\
  &= \{(1,4,1,1),(2,6,1,2),(2,6,2,1)\} ~ \cup ~ \left\{(1,2,1+\peg_1(D_1),1+\peg_2(D_1)),\right. \\
 & \qquad \left. (3,5,1+\peg_3(D_1),1+\peg_5(D_1)),(5,6,2+\peg_5(D_1),1+\peg_6(D_1))\right\} \\
  &= \{(1,4,1,1),(2,6,1,2),(2,6,2,1)\} ~\cup ~ \left\{(1,2,1+1,1+2),(3,5,1+0,1+0),\right.\\
& \qquad \left. (5,6,2+0,1+2)\right\} \\
  &= \{(1,4,1,1),(2,6,1,2),(2,6,2,1),(1,2,2,3),(3,5,1,1),(5,6,2,3)\}.
\end{align*}
\end{example}

A $k$-colouring $\alpha$ of a web diagram $D$ is an assignment of the
numbers $\{1,\ldots,k\}$ to the edges of $D$ such that the mapping
$\alpha:D \mapsto \{1,\ldots,k\}$ is surjective. In this instance we
write $|\alpha|=k$.  Let $\Colourings_k(W)$ be the set of
$k$-colourings of web diagrams in $W=W(D)$, and
$\Colourings(W)=\Colourings_1(W) \cup \Colourings_2(W) \cup \cdots$.

The construction procedure that we alluded to in the Introduction happens in the following way: a
$\ell$-colouring $\alpha$ of a web diagram $D$ produces $\ell$ new sub-web
diagrams $D_{\alpha}(1),\ldots,D_{\alpha}(\ell)$, each of which consists
of only the edges having that designated colour.  These $\ell$ diagrams
are not yet web diagrams, but we can relabel the heights of the
vertices on each of the pegs to make them into web diagrams.  Let
$\rel$ be this relabelling operation.  The new web diagram
$D'=\Recon{D}{\alpha}=\rel(D_{\alpha}(1))\oplus \cdots \oplus
\rel(D_{\alpha}(\ell))$ is formed by stacking the diagrams on top of one
another in increasing order of the colours.

The $(D_1,D_2)$ entries of the web-colouring and web-mixing matrices
of $W$ are:
\begin{align*}
\MMmatrix{(W)}{(x)}{D_1,D_2} = \sum_{\ell \geq 1} x^{\ell} f(D_1,D_2,\ell) 
\qquad
\mbox{ and }
\qquad
\RRmatrix{(W)}{D_1,D_2} = \sum_{\ell \geq 1} \dfrac{(-1)^{\ell-1}}{\ell} f(D_1,D_2,\ell),
\end{align*}
where $f(D_1,D_2,\ell)$ is the number of $\ell$-colourings $\alpha$ of
$D_1$ which give $\Recon{D_1}{\alpha}=D_2$.
As stated in the introduction, these web-mixing matrices occur in the
calculation of scattering amplitudes in QCD. More specifically, they
occur in the exponents of amplitudes, and dictate how the colour
charge and kinematic degrees of freedom of highly energetic quarks and
gluons are entangled by the radiation of additional lower energy
gluons.

\begin{example}\label{twopegsexamp}
Let $D_1=\{e_1=(1,2,1,2),e_2=(1,2,2,1)\}$.  The web world generated by
$D_1$ is $W(D_1)=\{D_1,D_2\}$ where
$D_2=\{e_1'=(1,2,1,1),e_2'=(1,2,2,2)\}$.  These are illustrated in
Figure~\ref{twopegs}.
\begin{figure}
\centerline{
\begin{tabular}{c@{\hspace*{3em}}c}
\includegraphics[scale=\retyup]{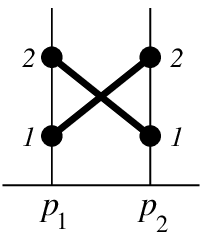} & \includegraphics[scale=\retyup]{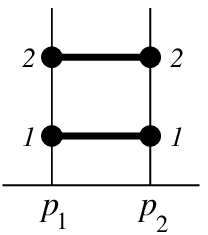}\\
$D_1$ & $D_2$
\end{tabular}
}
\caption{The two web diagrams of the web world in Example~\ref{twopegsexamp}.\label{twopegs}}
\end{figure}
There are three different colourings of $D_1$:
$$\begin{array}{l@{\quad}l@{\quad\Rightarrow\quad}l} \alpha (e_1)=1 &
  \alpha (e_2)=1 & \Recon{D_1}{\alpha}=D_1\\ \alpha (e_1)=1 & \alpha
  (e_2)=2 & \Recon{D_1}{\alpha}=D_2\\ \alpha (e_1)=2 & \alpha (e_2)=1
  & \Recon{D_1}{\alpha}=D_2
\end{array}$$
Consequently $\MMmatrix{(W)}{(x)}{D_1,D_1}=x^1$ and
$\MMmatrix{(W)}{(x)}{D_1,D_2}=2x^2$.  Likewise there are three
different colourings of $D_2$:
$$\begin{array}{ll@{\quad\Rightarrow\quad}l}
\alpha (e_1')=1 & \alpha (e_2')=1 & \Recon{D_2}{\alpha }=D_2\\
\alpha (e_1')=1 & \alpha (e_2')=2 & \Recon{D_2}{\alpha }=D_2\\
\alpha (e_1')=2 & \alpha (e_2')=1 & \Recon{D_2}{\alpha }=D_2
\end{array}$$
Consequently $\MMmatrix{(W)}{(x)}{D_2,D_1}=0$ and
$\MMmatrix{(W)}{(x)}{D_2,D_2}=x^1+2x^2$.  Therefore

$$\MMmatrix{(W)}{(x)}{} = 
\left(  \begin{matrix}
	x & 2x^2 \\
	0 & x+2x^2
	\end{matrix} \right)
\quad \mbox{ and }\quad
\RRmatrix{(W)}{} = 
\left(  \begin{matrix}
	1 & -1  \\
	0 & 0
	\end{matrix} \right).$$
\end{example}
The physical relevance of this example is that, owing to the second
row of this matrix having no non-zero entries, only diagram $D_1$
survives in the exponents of scattering amplitudes.  Further examples
of these definitions and details can be found in \cite[\S 2]{jcta}.

Every web diagram $D$ may be decomposed and written as a sum of
indecomposable web diagrams.  A poset (partially ordered set) $P$
on the set of these constituent indecomposable web diagrams can be defined as follows:

\begin{definition}
  Let $W$ be a web world and $D \in W$.  Suppose that $D=E_1 \oplus
  E_2 \oplus \cdots \oplus E_k$ where each $E_i$ is an indecomposable
  web diagram. 
  Let $P=\{E_1,\ldots,E_k\}$.
  Define the relation $<_1$ on $P \times P$ as follows:
  $E_i <_1 E_j$ if $i<j \mbox{ and } \PegSet(E_i)\cap \PegSet(E_j)\neq \emptyset$.
  Let $<_2$ be the transitive closure of $<_1$ on $P\times P$ and let $\preceq$ be the reflexive closure of $<_2$ on $P\times P$.
  We call $P(D)=(P,\preceq)$ the \emph{decomposition poset of $D$}.
\end{definition}

\begin{figure}
\centerline{\includegraphics{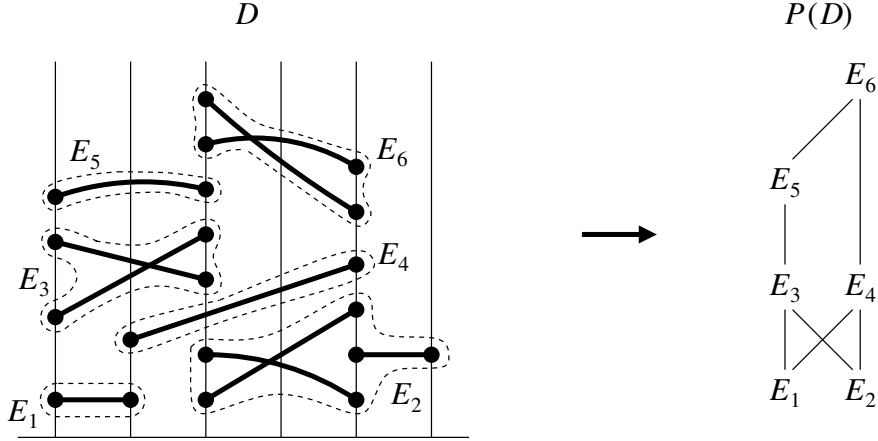}}
\caption{The decomposition poset of the web diagram in Example~\ref{wert}.\label{dec2examp}}
\end{figure}

\begin{example}\label{wert}
  Consider the web diagram $D$ given on the left in Figure~\ref{dec2examp}.  
  The poset $P(D)$ we get from this diagram illustrated to its right.
  The constituent indecomposable web diagrams are:
  $$\begin{array}{rcl} 
	E_1&=&\{(1,2,1,1)\} \\
  	E_2&=&\{(3,5,1,3),(3,5,2,1),(5,6,2,1)\}\\
  	E_3&=&\{(1,3,1,2),(1,3,2,1)\} \\ 
	E_4&=&\{(2,5,1,1)\}\\
  	E_5&=&\{(1,3,1,1)\}\\
	E_6&=&\{(3,5,1,2),(3,5,2,1)\}.
	\end{array}
	$$
\end{example}

In the following sections we will repeatedly come across the {\it{Fubini polynomials}} (also known as the {\it{ordered Bell polynomials}}) $\Fubini{n}(x) = \sum_{k=1}^n k! \stirlingsecond{n}{k} x^k$.


\section{A decomposition theorem for web matrices}\label{secthree}
The main result of this section is a theorem that explains how 
to form the web-colouring matrix of a web world that admits a decomposition as
a disjoint union of at least two other web worlds.
Physical arguments tell us that such webs do not contribute to the exponents of scattering amplitudes~\cite{Gardi:2010rn} and this is something we will discuss further and exploit in Section~\ref{secfour}.

In order to be able to state our main theorem, we must first define some new numbers 
and a new product on formal power series.

Let $\sNumb{k}{i_1,\ldots,i_m}$ be the number of 0-1 fillings of an 
$m$ row and $k$ column array such that there are precisely $i_1$ ones on the top row, 
$i_2$ ones in the second row, etc., and there are no columns of only zeros.

\begin{lemma}
$\sNumb{k}{i_1,\ldots,i_m}
=[u_1^{i_1} u_2^{i_2} \cdots u_m^{i_m}] \left((1+u_1)\cdots (1+u_m) - 1\right)^k $.
\end{lemma}
\begin{proof}
The generating function for an entry in row $j$ of the $m\times k$
array is $u_j^0+u_j^1 = 1+u_j$.  The only forbidden configuration in
the array is a column of all zeros, so the generating function for
columns is
$$\left(\prod_{j=1}^m (u_j^0+u_j^1) \right) - u_1^0 u_2^0\cdots u_m^0
= (1+u_1)\cdots (1+u_m) - 1.$$ Since there are no restrictions on
rows, the generating function for the $m\times k$ array is therefore
$$\left((1+u_1)\cdots (1+u_m) - 1\right)^k.$$ The number of arrays
with $i_1$ ones in the first row, $i_2$ ones in the second row, and so
on, is the coefficient of $u_1^{i_1} u_2^{i_2} \cdots u_m^{i_m}$ in the
above generating function:
\begin{align*}
\Numb{k}{i_1,\ldots,i_m} = [u_1^{i_1} u_2^{i_2} \cdots u_m^{i_m}]
\left((1+u_1)\cdots (1+u_m) - 1\right)^k.  
\end{align*}
\end{proof}

It is a simple exercise to check that
$\sNumb{k}{i_1} = \binom{k}{i_1}$ and 
$\sNumb{k}{i_1,i_2} = \binom{k}{k-i_1,k-i_2,i_1+i_2-k}$.
The general expression for $m=3$ is
\begin{align}
\Numb{k}{i_1,i_2,i_3} &= \binom{k}{i_1} \sum_{a=0}^{i_2} \binom{i_1}{a} \binom{k-i_1}{i_2-a} \binom{i_1+i_2-a}{k-i_3}.
\end{align}
We will now define a new product of power series that will both help in presenting our results concerning web worlds, but also serves in establishing a new automatic way to generate general combinatorial identities from web worlds.
The black diamond product $\webproduct$ is a binary operation that is both commutative and associative, and we define it in a very general form 
since this is how we will typically be using it.
\begin{definition}
Given $A^{(1)}(x),\ldots,A^{(m)}(x) \in \CC[[x]]$ where $A^{(k)}(x)=\sum_{n\geq 0} a^{(k)}_n x^n$,
we define the {\it{black diamond product}} of $A^{(1)}(x),\ldots,A^{(m)}(x)$ as:
$$A^{(1)}(x) \webproduct A^{(2)}(x) \webproduct \cdots \webproduct A^{(m)}(x)
= 
\sum_{k\geq 0} x^k
\sum_{i_1 \geq 0} \cdots \sum_{i_m \geq 0} 
a^{(1)}_{i_1} \cdots a^{(m)}_{i_m} 
\Numb{k}{i_1,\ldots,i_m}.$$
\end{definition}

Let us note that if the power series $A^{(i)}(x)$ are each polynomials of degree $n_i$, then the product may be written in the more computationally efficient form:

\begin{align}
A^{(1)}(x) \webproduct \cdots \webproduct A^{(m)}(x)
&=	
	\sum_{k=0}^{n_1+\cdots+n_m}
	x^k 
	\sum_{
	  \substack{
	  i_1 \in [0,n_1],\ldots, i_m \in [0,n_m] \\
	  i_1\vee\cdots\vee i_m \leq k \leq i_1+\cdots+i_m 
	  }}
	a^{(1)}_{i_1}\cdots a^{(m)}_{i_m} \Numb{k}{i_1,\ldots,i_m}.\label{easycomp}
\end{align}
The unit of the black diamond product is 1, i.e. $1 \webproduct A(x)=A(x)$.
To recap on our point before the definition, due to commutativity and associativity the order in which power series appear in the black diamond product does not change its outcome. 
In other words, given the power series $\{A^{(i)}(x)\}_{i \in [1,m]}$ and a permutation $\pi \in \Sym_m$, 
\begin{align}
A^{(\pi(1))}(x) \webproduct A^{(\pi(2))}(x) \webproduct \cdots \webproduct A^{(\pi(m))}
= A^{(1)}(x) \webproduct A^{(2)}(x) \webproduct \cdots \webproduct A^{(m)}(x).
\end{align}
We will abbreviate $\stackrel{m}{\overbrace{A(x) \webproduct \cdots \webproduct A(x)}}$ to $\webpower{A(x)}{m}$ with the convention that $\webpower{A(x)}{0}=1$.

\begin{example}\label{webprodexamp}
Suppose that $A^{(1)}(x)=\cdots=A^{(m)}(x)=x$. 
Then 
$$\webpower{x}{m} = \stackrel{m}{\overbrace{x \webproduct \cdots \webproduct x}} = \sum_{k\geq 0} x^k \Numb{k}{1,\ldots,1}$$
where there are $m$ ones on the right hand side. 
The value $\sNumb{k}{{1,\ldots,1}}$
is the number of ways to fill a table of $k$ columns and $m$ rows
with 0s and 1s such that there is exactly one 1 in every row and there are no columns of only 0s.
This is simply another way to encode an ordered set
partition of an $m$-set into $k$ sets (if $a$ is in the $i$th set of
such a sequence then the solitary 1 in row $a$ of the array will be
in the $i$th column). This number is $k! \stirlingsecond{m}{k}$ and
so for all $m \geq 1$,
\begin{align}
\webpower{x}{m} &= \sum_{k=1}^{m} x^k k! \stirlingsecond{m}{k} = \Fubini{m}{(x)}.\label{xm}
\end{align}
\end{example}

\begin{example}
Using the fact that $\sNumb{k}{i_1,i_2} = \binom{k}{k-i_1,k-i_2,i_1+i_2-k}$ and applying Equation~\ref{easycomp}, we have
$$(x+x^2) \webproduct (x+x^2) = x+7x^2+ 12x^3+6x^4.$$
Several of the power series that arise from taking the black diamond product of simple power series expressions appear to correspond to known sequences/power series. 
This suggests the black diamond product may be an object worthy of study in its own right. Some examples of these include:
$x \webproduct x^n = nx^n + (n+1)x^{n+1}$,
the coefficients in the power series $x^n \webproduct x^{n+1}$ appear to be given by ~\cite[A253283]{oeis}, and 
the sequence of coefficients of
$$x^n \webproduct  x^n = \sum_{k=0}^n \binom{n+k}{k} \binom{n}{k} x^{n+k}$$
is known to count several different structures (see \cite[A063007]{oeis}). 
(We omit the proof of this final observation since it is not immediately relevant to the paper's goal.)
\end{example}

We are now ready to state our main theorem.

\begin{theorem} \label{mixingthm}
Let $W_1,\ldots ,W_m$ be web worlds on pairwise disjoint peg sets.
Suppose that $D_i,D_i' \in W_i$ for all $i \in [1,m]$. 
Let $W=W_1\cup W_2 \cup \cdots \cup W_m$ be a new web world which is the disjoint union of $W_1,\ldots,W_m$.  
The diagrams $D=D_1\oplus \cdots \oplus D_m$ and $D'=D'_1 \oplus \cdots \oplus D'_m$ are web diagrams in $W$.
Then
$\MMmatrix{(W)}{(x)}{D,D'} \in \CC[[x]]$ where
\begin{align*}
\MMmatrix{(W)}{(x)}{D,D'} &= 
			\MMmatrix{(W_1)}{(x)}{D_1,D_1'} \webproduct \cdots \webproduct \MMmatrix{(W_m)}{(x)}{D_m,D_m'}.
\end{align*}
\end{theorem}

\begin{proof}
Suppose that $\alpha_1$ is an $i_1$-colouring of $D_1$ which
constructs $D_1'$ (i.e. $\Recon{D_1}{\alpha}=D_1'$) and that $\alpha_2$ is a $i_2$-colouring of $D_2$
that constructs $D_2'$, and so on.  
An $i$-chain is a totally ordered set of $i$ elements.  
The number of ways to
$k$-colour the diagram $D=D_1\oplus \cdots \oplus D_m$ so that it
becomes $D'$ is the number of ways to embed the $m$ chains
$i_1$-chain, $i_2$-chain, $\ldots$, $i_m$ chain, which represent the
ordered colourings of $D_1,\ldots ,D_m$, respectively, into a
$k$-chain.  

The value of $k$ must be at least the length of the
largest of the $m$ chains, and of course must be at most the sum of
the lengths of all constituent $m$ chains, i.e.  $k \in
[\max(i_1,\ldots,i_m),i_1+\cdots+i_m]$.  Furthermore, this embedding
must be surjective, for otherwise at least one of the $k$-colours
$\{1,\ldots,k\}$ would not have any corresponding colour in the $m$
constituent chains.
One can recast this in the form of a tabular 0-1 filling problem where we have 
$m$ rows and $k$ columns whereby 
the column indices of the $i_r$ ones that appear in row $r$ of the table indicate the new colours 
that they take on in the $k$-colouring.
The surjectivity condition translates into there being no columns of all zeros.
The number of ways to do this is therefore $\sNumb{k}{i_1,\ldots,i_m}$.

Suppose that the entries of the web-colouring matrices that correspond to the
diagrams are
\begin{align*}
\MMmatrix{(W_i)}{(x)}{D_i,D_i'} ~=~ a^{(i)}_1 x+\cdots +a^{(i)}_{n_i} x^{n_i} ~=:~ A^{(i)}(x)
\end{align*}
for all $i \in [1,m]$.  
Since there are $a^{(j)}_{i_j}$ ways to colour $D_j$ to produce
$D'_j$ for all $j \in [1,m]$, the factor $a^{(1)}_{i_1}\cdots a^{(m)}_{i_m}$
must be included.  Therefore
\begin{align*}
\MMmatrix{(W)}{(x)}{D,D'}
 &= \sum_{i_1=1}^{n_1} \cdots \sum_{i_m=1}^{n_m} 
			a^{(1)}_{i_1}\cdots a^{(m)}_{i_m} 
			\sum_{k=i_1\vee \cdots\vee i_m}^{i_1+\cdots+i_m}
				x^k \Numb{k}{i_1,\ldots,i_m}\\
&= 
\MMmatrix{(W_1)}{(x)}{D_1,D_1'} \webproduct \cdots \webproduct \MMmatrix{(W_m)}{(x)}{D_m,D_m'}.
\end{align*}
\end{proof}

\begin{example}\label{onebyone}
For all $i \in [1,m]$, let $D_i=\{(2i-1,2i,1,1)\}$ be the web diagram
that consists of a single edge between pegs $2i-1$ and $2i$, and let
$W_i$ be the web world that consists of the single web diagram $D_i$.
Then $\MMmatrix{(W_i)}{(x)}{} = (x)$, a $1\times 1$ matrix.  Since the
conditions of Theorem~\ref{mixingthm} hold, 
the web-colouring matrix of
$W=\{D=D_1\oplus \cdots \oplus D_m\}$, is
\begin{align*}
\MMmatrix{(W)}{(x)}{} &~=~ \left(\sum_{k=1}^{m} x^k \Numb{k}{1,\ldots,1}\right)
~=~ \left( \Fubini{m}{(x)}\right),
\end{align*}
from Example~\ref{webprodexamp}.
\end{example}

\begin{corollary} With the same setup as in Theorem~\ref{mixingthm},
\begin{align*}
\trace \MMmatrix{(W)}{(x)}{} &= 
	\sum_{D_1 \in W_1} \cdots \sum_{D_m \in W_m}
	\MMmatrix{(W_1)}{(x)}{D_1,D_1}
	\webproduct \cdots \webproduct
	\MMmatrix{(W_m)}{(x)}{D_m,D_m}.
\end{align*}
\end{corollary}


\section{A method for generating combinatorial identities}\label{secfour}
In this section we will outline a method for generating combinatorial identities 
by using the black diamond product in conjunction with a result concerning the trace of web-mixing matrices for disjoint web worlds.

\begin{proposition}\label{tracezero}
Let $W$ be a web world that is the disjoint union of at least two web worlds.
Then all entries of the web-mixing matrix $\RRmatrix{(W)}{}$ are zero, and consequently $\trace \RRmatrix{(W)}{}=0$.
\end{proposition}

A statistical physics-based proof of this result has
been given in Gardi et al.~\cite[Section 5.1]{Gardi:2010rn}. The
essential physical idea, elaborated further in~\cite{Gardi:2013ita},
is that exponents of scattering amplitudes can only contain
interactions between non-disjoint groups of quarks and gluons.

The relationship between the web-mixing matrices and web-colouring matrices is
\begin{align}\label{rtom}
\RRmatrix{(W)}{D,D'} = \int_{-1}^{0} \dfrac{\MMmatrix{(W)}{(x)}{D,D'}}{x} dx.
\end{align}
Integrating formal power series can be a contentious issue, so let
us be clear that the integrals we perform in this paper are definite integrals and always correspond to the
transformation: $$a_1x+a_2x^2+a_3x^3+a_4x^4+\cdots \quad\mapsto\quad a_1-a_2/2 + a_3/3 -a_4/4 +\cdots.$$

\newcommand{\Amatrix}{M}
Suppose that $W_1,\ldots,W_n$ are web worlds on disjoint peg sets that all have the same web-colouring matrix $\Amatrix = \MMmatrix{(W_i)}{(x)}{}$. 
(This happens for example when all web worlds are equivalent to one another by relabeling the peg sets that they are defined upon.) 
Suppose further that the diagonal entries of $\Amatrix$ are $(G_1(x),\ldots,G_t(x))$ and that $$\{G_1(x),\ldots,G_t(x)\}=\{H_1(x),\ldots,H_s(x)\}$$ where each of the $H_i(x)$ are distinct and have multiplicities $(h_1,\ldots,h_s)$ as diagonal entries in $\Amatrix$. 
Then
\begin{align}
\trace \MMmatrix{(W)}{(x)}{} 
&= 
	\sum_{D_1 \in W_1} \cdots \sum_{D_m \in W_m}
	\MMmatrix{(W_1)}{(x)}{D_1,D_1}
	\webproduct \cdots \webproduct
	\MMmatrix{(W_m)}{(x)}{D_m,D_m}\nonumber\\
&= 	\sum_{(i_1,\ldots,i_m) \in [1,t]^m}
	G_{i_1}(x) \webproduct \cdots \webproduct G_{i_m}(x)\nonumber\\
&= 	\sum_{(j_1,\ldots,j_m) \in [1,s]^m}
	h_{j_1} \cdots h_{j_m}
	H_{j_1}(x) \webproduct \cdots \webproduct H_{j_m}(x)\nonumber\\
&=	\sum_{\substack{a_1,\ldots,a_s \geq 0\\ a_1+\cdots+a_s=m}}
	h_1^{a_1} \cdots h_s^{a_s} \binom{m}{a_1,\ldots,a_s}
	\webpower{H_1(x)}{a_1} \webproduct \cdots \webproduct \webpower{H_s(x)}{a_s}.\label{tracegeneral}
\end{align}
For the case that there are only two different power series that appear as diagonal entries of $\Amatrix$, we have
\begin{align}
\trace \MMmatrix{(W)}{(x)}{}
&= \sum_{a=0}^m 
	h_1^a h_2^{m-a} \binom{m}{a}
	\webpower{H_1(x)}{a} \webproduct \webpower{H_2(x)}{m-a}. \label{tracetwo}
\end{align}

\begin{theorem}\label{genidents}
Let $W$ be a web world whose web-colouring matrix has 
$s$ different diagonal entries $(H_1(x),\ldots,H_s(x))$ that appear with 
multiplicities $(h_1,\ldots,h_s)$. Then for all positive integers $m$, we have
\begin{align*}
	\sum_{\substack{a_1,\ldots,a_s \geq 0\\ a_1+\cdots+a_s=m}}
	h_1^{a_1} \cdots h_s^{a_s} \binom{m}{a_1,\ldots,a_s}
	\int_{-1}^{0}
	\webpower{H_1(x)}{a_1} \webproduct \cdots \webproduct \webpower{H_s(x)}{a_s}
	\dfrac{dx}{x} &=0.
\end{align*}
The expression for the $s=2$ case is:
\begin{align*}
	\sum_{a=0}^m 
	h_1^a h_2^{m-a} \binom{m}{a}
	\int_{-1}^0 \webpower{H_1(x)}{a} \webproduct \webpower{H_2(x)}{m-a}
	\dfrac{dx}{x}
	=0.
\end{align*}
\end{theorem}

We will apply the above theorem to two extremely simple web worlds to see the combinatorial identities that emerge.

\begin{example}
Let $W$ be one of the web worlds of Example~\ref{onebyone} so that
$\MMmatrix{(W)}{(x)}{} = (x)$, a $1 \times 1$ matrix.
Applying Theorem~\ref{genidents} we have
$s=h_1=1$ and $H_1(x)=x$.
From Equation~\ref{xm} we have
$$\webpower{H_1(x)}{m} = \sum_{k=1}^m x^k k!\stirlingsecond{m}{k}$$
and so
\begin{align*}
\int_{-1}^0 \webpower{H_1(x)}{m} \dfrac{dx}{x} 
&=	\left[ \sum_{k=1}^m \dfrac{x^{k}}{k} k! \stirlingsecond{m}{k} \right]^0_{-1} \\
&=  \sum_{k=1}^m {(-1)^{k+1}} (k-1)! \stirlingsecond{m}{k} .
\end{align*}
This gives us the identity:
\begin{align}
\sum_{a_1=m} 1^{a_1} \binom{m}{a_1} \int_{-1}^0 \webpower{H_1(x)}{m} \dfrac{dx}{x} =
& \sum_{k=1}^m {(-1)^{k+1}} (k-1)! \stirlingsecond{m}{k} ~=~0.\label{firstnewidentity}
\end{align}
This identity can also be found in \cite[Eqn. 27]{wolfram}.
\end{example}

\begin{example}
Let $W$ be the web world of Example~\ref{twopegsexamp}, with
$$\MMmatrix{(W)}{(x)}{} =
\left(  \begin{matrix}
x & 2x^2 \\
0 & x+2x^2
\end{matrix} \right).$$

In this case $s=2$, $H_1(x)=x$, $H_2(x)=x+2x^2$, and $h_1=h_2=1$.
From Equation~\ref{xm}, we have $\webpower{H_1(x)}{m} =\Fubini{m}{(x)}$.
One can show that $\webpower{H_2(x)}{m}=\Fubini{2m}{(x)}$ and then
applying Theorem~\ref{genidents} we have
\begin{align*}
\webpower{H_1(x)}{a} \webproduct \webpower{H_2(x)}{m-a}
&= \sum_{k=0}^{2m-a} x^k \sum_{i_1,i_2} i_1! \stirlingsecond{a}{i_1} i_2!
	\stirlingsecond{2m-2a}{i_2} \Numb{k}{i_1,i_2} \\
&= \sum_{k=0}^{2m-a} x^k \sum_{i_1,i_2} i_1! \stirlingsecond{a}{i_1} i_2!
 \stirlingsecond{2m-2a}{i_2} \binom{k}{k-i_1,k-i_2,i_1+i_2-k}.
\end{align*}
This gives us the identity
\begin{align}
	\sum_{a=0}^m 
	\sum_{k=1}^{2m-a} 
	\sum_{i_1,i_2} 
	\binom{m}{a}
	\dfrac{(-1)^{k+1}}{k} 
	i_1! i_2! \stirlingsecond{a}{i_1}
 	\stirlingsecond{2m-2a}{i_2} \binom{k}{k-i_1,k-i_2,i_1+i_2-k}
	=0.
\end{align}
We do not know if this is a known identity.
\end{example}


\section{Repeated entries in web matrices}\label{secfive}
One observation that is apparent when calculating the web matrices (be
they web-mixing or web-colouring) of web worlds is that there are a
large number of entries that are the same. A deeper understanding of
this property would simplify the calculation of web-mixing matrices,
and has even led, for certain families of web diagram, to obtaining
the web-mixing matrix for arbitrary numbers of gluons~\cite{jhep}. In
this section we prove two new results to explain some of these
repetitive entries.  The first theorem gives one explanation for
repeated entries found on the diagonal of a web-colouring matrix, and also explains the same for the web-mixing matrix since the entries of the latter are a simple integral transformation of those in the web-colouring matrix.
This theorem also builds on our earlier result \cite[Thm. 3.4]{jcta}.

Given a poset $P=(P,\prec)$, its {\it{comparability graph}} $\comp(P)$
is the graph whose vertices are the elements of $P$, with $x,y \in P$
adjacent in $\comp(P)$ if $x \prec y$ or $y \prec x$.

\begin{theorem}
Let $D$ and $D'$ be web diagrams in a web world $W$ with 
\begin{align*}
D &= E_1 \oplus \cdots \oplus E_k \quad \mbox{and}\quad
D' = E'_1 \oplus \cdots \oplus E'_{k'},
\end{align*}
where each of the constituent diagrams $E_i$ and $E'_i$ are
indecomposable.
Suppose that the members of $(E_1,\ldots,E_k)$ are distinct and the members of
$(E'_1,\ldots,E'_{k'})$ are also distinct.  
If $\comp(P(D)) = \comp(P(D'))$, then
$\MMmatrix{(W)}{(x)}{D,D} = \MMmatrix{(W)}{(x)}{D',D'}$.
\end{theorem}

\begin{proof}
First let us consider the web diagram $D=E_1 \oplus \cdots \oplus E_k$
where each of the diagrams $E_i$ is indecomposable.  All
$\ell$-colourings $\alpha$ such that $\Recon{D}{\alpha}=D$ must have
the property that $\alpha (e)=\alpha (e')$ whenever $e,e' \in E_i$
because $E_i$ is indecomposable.  Thus $\alpha$ must be a surjective
map from $\{E_1,\ldots,E_k\}$ to $[1,\ell]$.  Equivalently $\alpha$
surjectively maps the $k$ elements of the decomposition poset $P(D)$ to
the total order on $[1,\ell]$.  Let us call the number of these
surjective maps $\Theta(P(D),\ell)$.  This number is related to the
number $\Omega(P(D),\ell)$ of order preserving maps from the poset
$P(D)$ to $[1,\ell]$ via an inclusion-exclusion argument:
\begin{align*}
\Omega(P(D),\ell) &= \sum_k \binom{\ell}{k} (-1)^{\ell-k} \Omega(P(D),\ell).
\end{align*}
(See ~\cite[Lemma 3.3]{jcta} for further details.)
Using this we have
\begin{align}
\MMmatrix{(W)}{(x)}{D,D} &= \sum_{\ell \geq 0} \Theta(P(D),\ell) x^{\ell} = \sum_{\ell \geq 0} x^{\ell} \sum_{k} \binom{\ell}{k} \Omega(P(D),k). \label{done}
\end{align}
The same is true for $D'$, so we have
\begin{align}
\MMmatrix{(W)}{(x)}{D',D'} &= \sum_{\ell \geq 0} x^{\ell} \sum_{k} \binom{\ell}{k} \Omega(P(D'),k).\label{dpone}
\end{align}
If $\comp(P(D))=\comp(P(D'))$ then by Stanley~\cite[Cor. 4.4]{stanley_dcg} we have that $\Omega(P(D))=\Omega(P(D'))$, and consequently the expressions in Equations~\ref{done} and ~\ref{dpone} are the same. 
\end{proof}

It is not immediately obvious that the previous theorem does indeed
help with enlarging the collection of diagrams that are known to have the same
diagonal entries in a web-colouring matrix. We give here a non-trivial
example, and one which cannot be discovered by our subsequent Theorem~\ref{riccar}.

\begin{example}\label{posethasseexample}
Let
$D=\{(1,2,1,1),(1,3,2,1),(1,4,3,1),(3,5,2,3),(5,6,2,1),(5,7,1,1)\}$
and
$D'=\{(1,2,1,1),(1,3,2,1),(1,4,3,1),(2,7,2,3),(6,7,1,2),(5,7,1,1)\}$.
The Hasse diagrams for $P(D)$ and $P(D')$ are illustrated in
Figure~\ref{Deed} and are clearly different.  However we have that
$\comp(P(D))=\comp(P(D'))=G$ and we can conclude that
$\MMmatrix{(W)}{(x)}{D,D}=\MMmatrix{(W)}{(x)}{D',D'}$.\\[1em]
\begin{figure}
\centerline{
\def\myscale{0.8}
\def\myother{1cm}
\begin{tabular}{c@{\hspace*{\myother}}c}
\includegraphics[scale=\myscale]{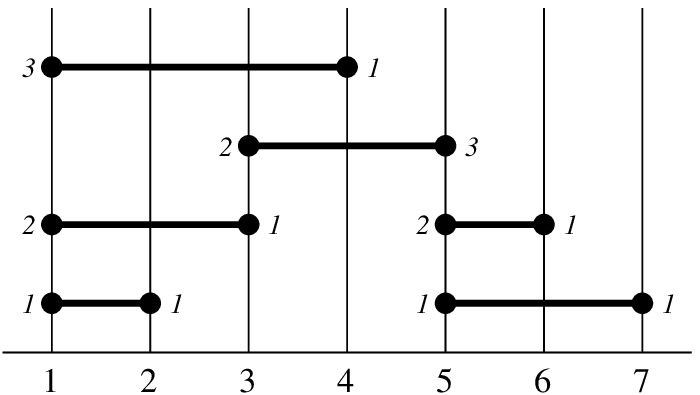} &
\includegraphics[scale=\myscale]{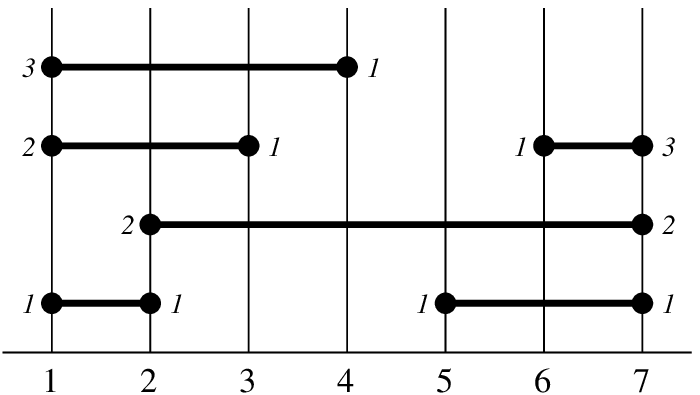} \\
$D$ & $D'$ \\[1em]
\includegraphics[scale=\myscale]{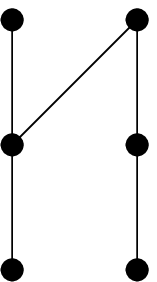} &
\includegraphics[scale=\myscale]{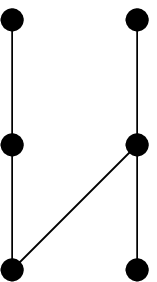} \\
$P(D)$ & $P(D')$\\[0em]
\multicolumn{2}{c}{\includegraphics[scale=\myscale]{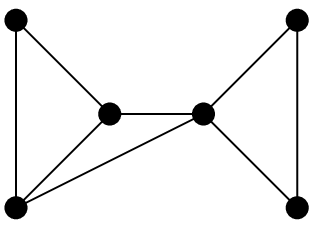}} \\
\multicolumn{2}{c}{$G$}
\end{tabular}
}
\caption{The web diagrams, their deconstruction posets, and web graph for Example~\ref{posethasseexample}.\label{Deed}}
\end{figure}
\end{example}

Our second result addresses general entries of the web-colouring matrix and
upside-down web diagrams.  Given a web diagram $D$, let $\flip(D)$ be
the web diagram achieved by turning it upside-down. This operation is,
by definition, an involution.

\begin{definition}
Let $D=\left\{e_i=(a_i,b_i,c_i,d_i)~:~ 1\leq i \leq m\right\}$ be a
web diagram in a web world $W$.  The flip of $D$, $\flip(D)$ is the
web diagram in $W$ with edges:
$$\flip(D) =\left\{(a_i,b_i,p_{a_i}+1-c_i,p_{b_i}+1-d_i)~:~ 1\leq i\leq m \right\}.$$
\end{definition}

\begin{example}
Let $D$ be the web diagram in Example~\ref{posethasseexample} and
illustrated in Figure~\ref{Deed}.
Then $$\flip(D)=\{(1,2,3,1),(1,3,2,2),(1,4,3,1),(3,5,1,1),(5,6,2,1),(5,7,3,1)\}.$$
\centerline{ \def\myscale{0.6} \def\myother{8ex}
\begin{tabular}{c}
\includegraphics[scale=\myscale]{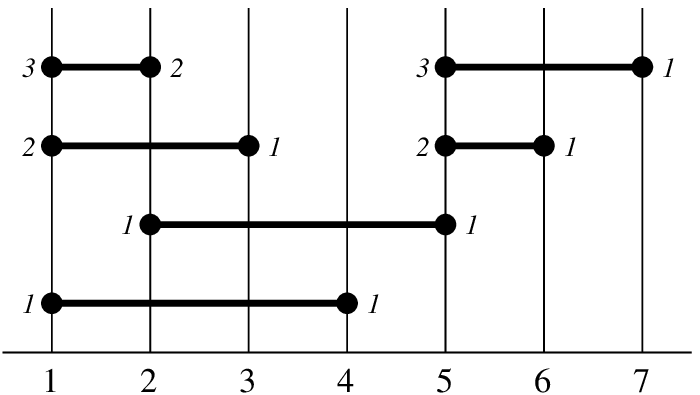} \\
$\flip(D)$
\end{tabular}
}
\end{example}

\begin{theorem}\label{riccar}
Let $W$ be a web world and $D, D' \in W$.
Then $\MMmatrix{(W)}{(x)}{D,D'} = \MMmatrix{(W)}{(x)}{\flip(D),\flip(D')}$.
\end{theorem}

\begin{proof}
Let $W$ be a web world and $D, D'$ diagrams in $W$.  To prove the
theorem it suffices to show that for every $k$-colouring $\alpha$ of
$D$ that yields $D'$, there is a unique $k$-colouring $\beta$ of
$\flip(D)$ that yields $\flip(D')$.  Define the colouring $\beta$ as
follows: if $e=(a,b,x,y) \in D$ with $\alpha (e)=\ell$, then let
$\beta (e)=k+1-\ell$.

We will now show that $\Recon{\flip(D)}{\beta } = \flip(D')$ which
implies the result.  Notice that in general if $B_1,\ldots , B_m$ are
web diagrams then
\begin{align} \label{flipidentity}
\flip(B_1\oplus \cdots \oplus B_m) = \flip(B_m) \oplus \cdots \oplus \flip(B_1).
\end{align}
We have:
\begin{align*}
\Recon{\flip(D)}{\beta } &= \rel(\flip(D)_{\beta }(1)) \oplus \cdots \oplus \rel(\flip(D)_{\beta }(k)) \\
&= \flip(\rel(D_{\beta }(1))) \oplus \cdots \oplus \flip(\rel(D_{\beta }(k))).
\end{align*}
Using Equation~\ref{flipidentity}, and applying $\flip$ to both sides of the previous equation, we have
\begin{align*}
\flip(\Recon{\flip(D)}{\beta })&= \flip(\flip(\rel(D_{\beta }(1))) \oplus \cdots \oplus \flip(\rel(D_{\beta }(k))))\\
&= \flip(\flip(\rel(D_{\beta }(k)))) \oplus \cdots \oplus \flip(\flip(\rel(D_{\beta}(1))))\\
&=\rel(D_{\beta}(k)) \oplus \cdots \oplus\rel(D_{\beta}(1)).
\end{align*}
Since those edges of $D$ that are coloured $\ell$ using the colouring
$\beta$ are precisely the same as the edges that are coloured
$k+1-\ell$ using the colouring $\alpha$, we have $D_{\beta}(\ell) =
D_{\alpha}(k+1-\ell)$ and
\begin{align*}
\flip(\Recon{\flip(D)}{\beta})&= \rel(D_{\alpha}(1)) \oplus \cdots \oplus \rel(D_{\alpha}(k))\\
&= D'.
\end{align*}
Applying $\flip$ to both sides gives $\Recon{\flip(D)}{\beta} = \flip(D')$, as was required.
\end{proof}


\section{The squares of web-colouring and web-mixing matrices}\label{secsix}
In this section we show how to calculate the entries for the square of
a web-colouring matrix and give the first mathematical proof that web-mixing
matrices are idempotent. This property has been established using
physical arguments by Gardi and White~\cite[\S 3]{GPO}, and plays an
important role in QCD calculations. Idempotence implies, for example,
that all web-mixing matrices have eigenvalues zero or one (with some
multiplicity). Only combinations of web diagrams in a given web world
associated with unit eigenvalues of the mixing matrix survive in the
exponent of the scattering amplitude. 
\begin{theorem}\label{thm41}
Let $W$ be a web world whose diagrams each have $n$ edges.  Let $D,D'
\in W$ and suppose that $\MMmatrix{(W)}{(x)}{D,D'} =\bbbeta_1 x+\cdots
+\bbbeta_n x^n$.  Then
\begin{align*}
\left( \MMmatrix{(W)}{(x)}{} \right)^2_{D,D'}
	& = \sum_{i=1}^n  \bbbeta_i L_i(x)
\end{align*}
where $\ds
L_i(x) = \sum_{j,k\geq 1} x^{j+k} 
	\sum_{b=0}^j
	\sum_{a=0}^k
	(-1)^{j+k-(b+a)} \binom{j}{b} \binom{k}{a} \binom{ab}{i}$. 
(See Figure~\ref{litable}.)
\end{theorem}

Figure~\ref{litable} presents the expressions for $L_n(x)$ for all $n
\in [1,7]$.  The previous theorem may be summarised by saying that the
square of the web-colouring matrix $\MMmatrix{(W)}{(x)}{}^2$ is the
image of $\MMmatrix{(W)}{(x)}{}$ under the linear operator
$T:\Complex{[[x]]} \to \Complex{[[x]]}$ which maps the basis
$(x^i)_{i\geq 0}$ for $\Complex{[[x]]}$ to $(L_i(x))_{i\geq 0}$.

\begin{figure}
$$\begin{array}{|c|c|} \hline
n & L_n(x) \\ \hline\hline
0 & 1 \\
1 & x^2\\
2 & 2x^3+2x^4\\
3 & 6x^4+12x^5+6x^6\\
4 & x^4 + 26x^5 + 73x^6 + 72x^7 + 24x^8\\
5 & 12x^5 + 156x^6 + 516x^7 + 732x^8 + 480x^9 + 120x^{10}\\
6 & 2x^5+  126x^6 + 1206x^7 + 4322x^8 + 7680x^9 + 7320x^{10} + 3600x^{11} + 720x^{12}\\ \hline
\end{array}$$
\caption{The first few power series $L_n(x)$.\label{litable}}
\end{figure}

\begin{proof}
Let $\alpha  \in \Colourings(W)$.
Given two colourings $\alpha, \alpha ' \in \Colourings(W)$ there is a unique
$\gamma \in \Colourings(W)$ such that
$\Recon{D}{\gamma}=\Recon{\Recon{D}{\alpha}}{\beta}$.  In such an
instance we write $\gamma \equiv \beta \circ \alpha$.  Recall that the
$(D,D')$ entry of the web-colouring matrix of $W$ is:
\begin{align*}
\MMmatrix{(W)}{(x)}{D,D'}
	&= \sum_{\substack{\alpha \in \Colourings(W) \\ \Recon{D}{\alpha}=D'}} x^{|\alpha|} 
\end{align*}
and suppose that this is 
$\MMmatrix{(W)}{(x)}{D,D'}=\bbbeta_1 x+\cdots+\bbbeta_n x^n$.  Then
\begin{align*}
\left(
\MMmatrix{(W)}{(x)}{}
\right)^2_{D,D'}
&= 
\sum_{D'' \in W} \MMmatrix{(W)}{(x)}{D,D''} \MMmatrix{(W)}{(x)}{D'',D'} \\
&=
\sum_{D'' \in W}
\sum_{\substack{\alpha \in \Colourings(W) \\ \Recon{D}{\alpha}=D''}} x^{|\alpha|}
\sum_{\substack{\beta  \in \Colourings(W) \\ \Recon{D''}{\beta}=D'}} x^{|\beta|}
\\
&=
\sum_{j,k=1}^n \sum_{\substack{\alpha \in \Colourings_j(W)\\ \beta \in \Colourings_k(W)\\ \Recon{\Recon{D}{\alpha}}{\beta}=D' }} x^{j+k}\\
&=
\sum_{j,k=1}^n
\sum_{i=1}^n
 \sum_{\substack{\gamma \in \Colourings_i(W) \\ \Recon{D}{\gamma}=D'}}
  \sum_{\substack{\alpha\in \Colourings_j(W)\\ \beta \in \Colourings_k(W) \\ \gamma \equiv \beta \circ \alpha}} x^{j+k}\\
&=
\sum_{i=1}^n
 \sum_{\substack{\gamma \in \Colourings_i(W) \\ \Recon{D}{\gamma}=D'}}
\sum_{j,k=1}^n
	\OtherNumb{i}{j,k}
	x^{j+k}
=
\sum_{i=1}^n  \bbbeta_i \sum_{j,k=1}^n  \OtherNumb{i}{j,k} x^{j+k},
\end{align*}
where $\OtherNumb{i}{j,k}$ is the number of pairs $(\beta,\alpha)$ of
colourings in $\Colourings_j(W) \times \Colourings_k(W)$ such that
$\beta \circ \alpha$ is equivalent to a $i$-colouring $\gamma \in
\Colourings_i(W)$.
To see how a colouring of a colouring of a diagram can be turned into
one colouring~\footnote{The argument used here is reminiscent of the double application of the statistical physics-based {\it replica trick} in ref.~\cite{GPO}.}, we consider the ordered pairs of colours.  The diagram
$D$ is coloured with $\alpha$ and reconstructed to give the diagram
$D''$.  The diagram $D''$ is coloured with $\beta$ and reconstructed
to give the diagram $D'$.  

Let us consider what happens to an edge
$e_i=(a_i,b_i,x_i,y_i) \in D$ during this process.  In the first
colouring, $\alpha (e_i)=v$, say, and this edge in $D''$ will be
represented as $e_i'=(a_i,b_i,x_i',y_i')$ where $x_i'$ and $y_i'$ are
the new heights of the endpoints of that edge in $D''$.  The edge
$e_i'$ is now coloured $\beta (e_i')=u$, say, and this edge will be
represented as $e_i''=(a_i,b_i,x_i'',y_i'')$ in $D'$ where $x_i''$ and
$y_i''$ are the heights of the endpoints of $e_i''$ in $D'$.  Although
the heights change from one diagram to the next, we can still index
the edge in question throughout as the $i$th edge $e_i$ and write
$\alpha(e_i)=v$ and $\beta(e_i)=u$.

Let $D \in W$ and $(\alpha,\beta) \in \Colourings_j(W)\times
\Colourings_k(W)$.  Let us define $$D_{\beta,\alpha}(u,v) := \{e \in D
~:~ \beta(e)=u \mbox{ and } \alpha(e)=v\}.$$ Then
\begin{align*}
D' &= \rel(D_{\beta,\alpha }(1,1)) \oplus \cdots \oplus \rel(D_{\beta,\alpha}(1,k)) \oplus \cdots \oplus \rel(D_{\beta,\alpha}(j,1)) \oplus \cdots \oplus \rel(D_{\beta,\alpha}(j,k)).
\end{align*}
In other words, the subweb-diagram of $\Recon{D}{\alpha}$ whose edges
$e$ have colour $\beta(e)=1$ will maintain their order relative to one
another in $D'$ and be the lowest.  But amongst these edges, which
will be the lowest w.r.t. the colouring $\alpha$ of $D$?  It will be
those edges $e \in D$ that are coloured $\alpha(e)=1$, followed by all
those edges $e \in D$ with $\alpha(e)=2$, and so on.  Thus the
lexicographic order on all such pairs of colourings that exists
corresponds to the total order for a single colouring.

For example, if we have the pairs of colourings
$(\beta(e),\alpha(e))=(1,3)$, $(1,5)$, $(2,1)$, $(2,2)$, $(2,5)$, $(3,4)$, then these
correspond to the colours $\gamma(e)=1,2,3,4,5,6$, respectively.

The number $\OtherNumb{i}{j,k}$ is therefore the number of $j\times k$
zero-one matrices that have exactly $i$ ones and there are neither
columns nor rows of only zeros.  This number is known to be (see
\cite{peter,peter_GM} or ~\cite[Eqn. 120]{Maia}):
\begin{align*}
\OtherNumb{r}{m,n} &= \sum_{\ell \geq r} \sum_{d|\ell} (-1)^{n+m-(d+\ell/d)} 
		\binom{m}{d}\binom{n}{\ell/d} \binom{\ell}{r}\\
&= 
\sum_{b=0}^m  
\sum_{a=0}^n 
(-1)^{n+m-(b+a)} \binom{m}{b} \binom{n}{a} \binom{ab}{r}
\end{align*}
Therefore
$
\left(
\MMmatrix{(W)}{(x)}{}
\right)^2_{D,D'}
= \ds\sum_{i=1}^n  \bbbeta_i L_i(x)
$
where 
\begin{align*}
\ds
L_i(x) = \sum_{j,k\geq 1} x^{j+k} 
	\sum_{\substack{a\in[0,k]\\ b\in [0,j]}}
	(-1)^{j+k-(b+a)} \binom{j}{b} \binom{k}{a} \binom{ab}{i}.
\end{align*}
\end{proof}

As mentioned above, the following theorem has already been proven
using a physics argument (see ~\cite[Sect. 3]{GPO}).  Here we give the first
combinatorial proof.

\begin{theorem} 
Let $W$ be a web world whose web diagrams each have $n$ edges. The
web-mixing matrix $\RRmatrix{(W)}{}$ is idempotent; $(\RRmatrix{(W)}{}
) ^2=\RRmatrix{(W)}{} $.
\end{theorem}

\begin{proof}
\newcommand{\TT}{\mathcal{T}}
Using the same method of summation as we did at the start of the proof
of Theorem~\ref{thm41},
\begin{align*}
(\RRmatrix{(W)}{})^2_{D,D'} 
&= \sum_{D'' \in W} \RRmatrix{(W)}{D,D''} \RRmatrix{(W)}{D'',D'}  \\
&= \sum_{i,j =1}^n \sum_{\substack{\alpha \in \Colourings_i(W) \\ \beta \in \Colourings_j(W)\\ \Recon{\Recon{D}{\alpha}}{\beta}=D'}} \dfrac{(-1)^{i-1}}{i} \dfrac{(-1)^{j-1}}{j} \\
&= \sum_{i,j =1}^n \dfrac{(-1)^{i-1}}{i} \dfrac{(-1)^{j-1}}{j}  
	\sum_{\substack{\alpha \in \Colourings_i(W) \\ \beta \in \Colourings_j(W)\\ \Recon{\Recon{D}{\alpha}}{\beta}=D'}} 1.
\end{align*}
Since 
\begin{align*}
\sum_{\substack{\alpha \in \Colourings_i(W) \\ \beta \in \Colourings_j(W)\\ \Recon{\Recon{D}{\alpha}}{\beta}=D'}} 1
&= \sum_{k\geq 1} \sum_{\substack{\gamma \in \Colourings_k(W)\\ \gamma(D)=D'}} 	
	\sum_{\substack{\alpha \in \Colourings_i(W) \\ \beta \in \Colourings_j(W)\\ \beta\circ \alpha \equiv \gamma}} 1
= \sum_{k\geq 1} \sum_{\substack{\gamma \in \Colourings_k(W)\\ \gamma(D)=D'}} \OtherNumb{k}{i,j},
\end{align*}
this gives
\begin{align}
(\RRmatrix{(W)}{})^2 _{D,D'}
&= \sum_{i,j =1}^n \dfrac{(-1)^{i-1}}{i} \dfrac{(-1)^{j-1}}{j}
	\sum_{k \geq 1} \sum_{\substack{\gamma \in \Colourings_k(W)\\ \gamma(D)=D'}} \OtherNumb{k}{i,j} \nonumber \\
&= \sum_{k=1}^{n} \sum_{\substack{\gamma \in \Colourings_k(W)\\ \gamma(D)=D'}} (-1)^{-2} \sum_{i,j=1}^{n} \dfrac{(-1)^{i+j}}{ij} \OtherNumb{k}{i,j}. \label{chess}
\end{align}
The inner sum is 
\begin{align}
\sum_{i,j=1}^{n} \dfrac{(-1)^{i+j}}{ij} \OtherNumb{k}{i,j}
&= \sum_{i,j=1}^{n} \dfrac{(-1)^{i+j}}{ij} \sum_{a=0}^i \sum_{b=0}^j (-1)^{i+j-(a+b)} \binom{i}{a} \binom{j}{b} \binom{ab}{k} \nonumber \\
&= \sum_{i,j=1}^{n} \sum_{a=0}^i \sum_{b=0}^j (-1)^{2(i+j)-(a+b)}\dfrac{1}{i} \binom{i}{a} \dfrac{1}{j} \binom{j}{b} \binom{ab}{k}\nonumber \\
&= \sum_{i,j=1}^{n} \sum_{a=0}^i \sum_{b=0}^j (-1)^{a+b}\dfrac{1}{k} \binom{i-1}{a-1} \binom{j-1}{b-1} \binom{ab-1}{k-1} \nonumber\\
&= \dfrac{1}{k} \sum_{a,b=1}^n (-1)^{a+b} \binom{ab-1}{k-1} \sum_{i=a}^n \sum_{j=b}^{n} \binom{i-1}{a-1} \binom{j-1}{b-1} \nonumber \\
&= \dfrac{1}{k} \sum_{a,b=1}^n (-1)^{a+b} \binom{ab-1}{k-1} \binom{n}{a} \binom{n}{b}.\label{rteqn}
\end{align}
We now need to prove that for $1\leq k \leq n$,
$$F(n,k)~ :=~ \sum_{a,b=1}^n (-1)^{a+b} \binom{ab-1}{k-1} \binom{n}{a} \binom{n}{b}=(-1)^{k-1}.$$
Let $G(n,k,a)=\sum_{b=1}^{n} (-1)^b \binom{n}{b} \binom{ab-1}{k-1}$ so that $F(n,k)=\sum_{a=1}^{n} (-1)^a \binom{n}{a} G(n,k,a)$.
Now
\begin{align*}
\binom{ab -1}{k-1} &= \dfrac{(ab)(ab-1)\cdots (ab-k+1)}{ab(k-1)!} \\
&=\dfrac{1}{(k-1)!} \sum_{j=1}^k \stirlingfirst{k}{j} (-1)^{k-j} (ab)^{j-1},
\end{align*}
the second equality comes from Graham et
al.~\cite[eqn. (6.13)]{knuth_concrete} and $\stirlingfirst{k}{j}$ are
the Stirling numbers of the first kind.  Replacing this into the
expression 
\begin{align*}
G(n,k,a) 
	&= \sum_{b=1}^{n} (-1)^b \binom{n}{b} \dfrac{1}{(k-1)!} \sum_{j=1}^k \stirlingfirst{k}{j} (-1)^{k-j} (ab)^{j-1}\\
	&= 
	\dfrac{1}{(k-1)!} \sum_{j=1}^k \stirlingfirst{k}{j} (-1)^{k-j} a^{j-1} 
	\sum_{b=1}^{n} (-1)^b \binom{n}{b} b^{j-1}.
\end{align*}
Using Graham et al.~\cite[Eqn. (6.19)]{knuth_concrete}: $m! \stirlingsecond{n}{m} = \sum_k \binom{m}{k} k^n (-1)^{m-k}$ where $\stirlingsecond{n}{m}$ are the Stirling numbers of the second kind. The inner sum in the previous expression is
\begin{align*}
\sum_{b=1}^{n}(-1)^b \binom{n}{b} b^{j-1} = n! \stirlingsecond{j-1}{n}(-1)^n - 0^{j-1}.
\end{align*}
Since $k \in [1,n]$, and $j\leq k$, the Stirling number on the right
hand side of the previous equation will always be zero for the values
that we are summing over, and the only term that will contribute will
be $-0^{j-1}$. This term itself will always be zero, except for when
$j-1=0$. This gives
\begin{align*}
G(n,k,a) &= \dfrac{1}{(k-1)!} \sum_{j=1}^k \stirlingfirst{k}{j} (-1)^{k-j} a^{j-1} (-0^{j-1}) \\
&= \dfrac{1}{(k-1)!} \stirlingfirst{k}{1} (-1)^{k-1} a^0 (-0^0) \\
&= (-1)^k.
\end{align*}
Therefore $F(n,k)=\sum_{a=1}^{n} (-1)^a \binom{n}{a} (-1)^k = (-1)^k
((1-1)^n-1) = (-1)^{k-1}.$ This allows us to write
Equation~\ref{rteqn} as $(-1)^{k-1}/k$ which, in turn, means that
Equation~\ref{chess} is
\begin{align*}
(\RRmatrix{(W)}{})^2 _{D,D'}
&= \sum_{k=1}^{n} \sum_{\substack{\gamma \in \Colourings_k(W)\\ \gamma(D)=D'}} \dfrac{(-1)^{k-1}}{k} 
= \RRmatrix{(W)}{D,D'}.
\end{align*}
\end{proof}


\section{Indecomposable permutations and a web world on two pegs having multiple edges}\label{secseven}
In this section we consider the web world
$W_n$ whose web graph is $G(W_n)=K_2$ and the single edge is labelled with $n$.
This web world consists of diagrams having two pegs and $n$ edges connecting the
$n$ vertices on the first peg to $n$ vertices on the second peg. From
a physics point of view, such webs have been studied for a long
time~\cite{Gatheral:1983cz,Frenkel:1984pz,Sterman:1981jc}. However,
they have yet to be analysed using the language of web-colouring and
mixing matrices and which we will now attend to.

Every diagram $D \in W_n$ is uniquely described as a set of 4-tuples
$$D=\{(1,2,i,\pi_i) ~:~ \pi \in \Sym_n\}$$ and consequently there are
$n!$ diagrams in $W_n$.  We can abbreviate this by writing
$D=D_{\pi}$. Given a permutation $\pi =\pi_1\ldots \pi_n$ and
permutations $\alpha \in \Sym_{\ell}$, $\beta \in \Sym_{n-\ell}$, we
say that $\pi$ is the sum of the permutations $\alpha$ and $\beta$ if
$\pi_i=\alpha_i$ for all $i \leq \ell$ and $\pi_i = \ell +
\beta_{i-\ell}$ for all $i > \ell$. We write this as
$\pi=\alpha+\beta$.

If a permutation $\pi$ can be written as such a sum of smaller permutations, then we say that $\pi$ is decomposable. 
Otherwise we say that $\pi$ is indecomposable.
Let $\Inde_n$ be the set of indecomposable permutations of length $n$ and let $\Inde=\cup_{n \geq 1} \Inde_n$.
Every permutation $\pi \in \Sym_n$ admits a unique representation as a sum of indecomposable permutations.
If $\pi = \sigma_1 + \cdots +\sigma_k$ where every $\sigma_i \in \Inde$, then $D_{\pi}=D_{\sigma_1} \oplus \cdots \oplus D_{\sigma_k}$ and we will write $\indparts(\pi)$ for the number $k$ of indecomposable permutations in the sum.

\begin{problem}{[Word reconstruction generating function]}
\label{problemone}
\ \\
Given a finite word $w=w_1\ldots w_n$ on a finite alphabet,
let $f_w(k)$ be the number of ways to read the word from left to right $k$ times (each time always picking some unchosen letter) so 
that we arrive back at the original word $w$. 
Let $F_w(x) = \sum_{k} f_w(k) x^k$.
\end{problem}

\begin{example}
Let $w=cbabac=w_1w_2w_3w_4w_5w_6$ be a word. 
The number of ways to read $w$ from left to right in one passing is 1 since we must sequentially read all letters $w_1w_2w_3w_4w_5w_6$, so $f_w(1)=1$.
There are several ways to read $w$ from $w$ in two passings: read positions 1 through to 5 on the first passing, and read position 6 on the second passing.
This is summarized by the reading vector $(1,1,1,1,1,2)$ where $j$ at position $i$ means $w_i$ is read on the $j$th passing.
If we abbreviate this to 111112, then the other reading vectors for $k=2$ are
111122,
111222,
112222,
122222, and
122112
so $f_w(2)=6$.
For this case, $F_w(x)= x+ 6x^2 + 17x^3 +26 x^4 + 22x^5 + 8 x^6$.
\end{example}

\begin{example}
If all letters of a length $n$ word $w$ are distinct, then $F_w(x)=x(1+x)^{n-1}$. 
Similarly if $w$ is a word consisting of $n$ copies of the same letter, then $F_w(x) = \sum_{k=1}^{n} k! \stirlingsecond{n}{k} x^k=\Fubini{n}{(x)}$.
\end{example}

\begin{theorem}
Let $D_{\pi} \in W_n$ with $\pi=w_1\oplus \cdots \oplus w_n$ where every $w_i \in \Inde$.
Let $w$ be the word $(w_1,\ldots,w_{\indparts(\pi)})$ on the alphabet $\{w_1,\ldots,w_{\indparts(\pi)}\}$.
Then $\MMmatrix{(W)}{(x)}{D_{\pi},D_{\pi}} =  F_w(x)$.
\end{theorem}

\begin{proof}
Let $D_{\pi} \in W_n$ with $\pi=\sigma_1\oplus \cdots \oplus\sigma_m$. 
In order to colour the edges of $D_{\pi}$ so that the same diagram is reconstructed, 
the colours of the edges that are in the same indecomposable part must remain the same. 
(Otherwise we would be constructing a web diagram which has more indecomposable parts than the original one, which would be a contradiction.)
\newcommand{\Col}[2]{\mathsf{Col}_{#1}(#2)}
Given a set $A=\{a_1,\ldots,a_i\}$, define $\sigma_A = \sigma_{a_1}\oplus \cdots \oplus\sigma_{a_i}$ where $a_1< \cdots <a_i$.
Let $$\RestColourings{m}{k} = \{ (c_1,\ldots,c_m)\in [1,k]^m~:~ \{c_1,\ldots,c_m\}=[1,k]\}.$$
Given $c \in \RestColourings{m}{k}$, let $\Col{c}{i}= \{j \in [1,m] ~:~ c_j=i\}$.

The set of colourings $c'$ of edges of $D_{\pi}$ which results in a reconstruction of the diagram is in 1-1 correspondence with 
the set of colourings $c\in \RestColourings{m}{\cdot}$ such that
$$\sigma_1 \oplus \cdots \oplus \sigma_m = \sigma_{\Col{c}{1}} \oplus \cdots \oplus \sigma_{\Col{c}{k}}.$$
This is the word reconstruction problem (Problem~\ref{problemone}) for the word $w=(\sigma_1,\ldots,\sigma_m)$ which is on the alphabet $\{\sigma_1,\ldots,\sigma_m\}$.
\end{proof}

There is no simple formula for the diagonal entries of the web-colouring matrix for this web world due to there being no 
`closed form' answer to Problem~\ref{problemone}.
However, we have noticed some potential structure from looking at the trace of the web-colouring matrices for small values of $n$, outlined in Figure~\ref{potential}.
The numbers from Figure~\ref{potential} support the following conjecture for all $n\in [1,8]$.

\begin{figure}
$$
\begin{array}{|c|c|} \hline
n & \trace \left(\MMmatrix{(W(n))}{(x)}{}\right) \\ \hline \hline
1 & x \\
2 & 2x+2x^2 \\
3 & 6x+8x^2+6x^3 \\
4 & 24x+30x^2+42x^3+24x^4 \\
5 & 120x+116x^2+216x^3+264x^4+120x^5 \\
6 & 720x+532x^2+1002x^3+1920x^4+1920x^5+720x^6 \\
7 & 5040x+2848x^2+ 4626x^3+ 11688x^4+ 19200x^5+15840x^6+5040x^7 \\  \hline
\end{array}
$$
\caption{The sequence of numbers that correspond to \cite[A121635]{oeis} (as mentioned in Conjecture~\ref{conjme}) are the coefficients of the second highest exponent of $x$ in each row: 2, 8, 42, 264, 1920 and 15840.
\label{potential}}
\end{figure}

\begin{conjecture}\label{conjme}
Let $a(n)=[x^{n-1}] \trace \left(\MMmatrix{(W(n))}{(x)}{}\right)$. Then $a(n)=(n-2)!(n^2-3n+4)/2$ for all $n>1$.
This sequence of numbers enumerates a class of directed column convex polyominoes in the plane. (See ~\cite[A121635]{oeis} and ~\cite{barcucci}.)
\end{conjecture}

\end{document}